\UseRawInputEncoding
\documentclass[11pt]{article}
\usepackage{graphicx,amssymb,amsmath,amsthm,amsfonts,color,booktabs,subfigure,float,epstopdf,mathrsfs}
\usepackage{mathrsfs}
\usepackage{leftidx}
 \usepackage[bookmarks=false]{hyperref}
\usepackage{amsmath,amssymb,amsfonts,amsthm,bm,mathrsfs}
\theoremstyle{plain}
\newtheorem{theorem}{Theorem}[section]
\newtheorem{lemma}{Lemma}[section]

\numberwithin{equation}{section}

\theoremstyle{definition}

\newtheorem{remark}{Remark}[section]

\def\cR{{\cal R}}

\def\Re {\cR e }

\begin{document}

\author{ Hai-E Zhang$^{1,2}$\qquad
Gen-Qi Xu$^{1}$\qquad Zhong-Jie Han$^{1}$\thanks{
Z. J. Han(Corresponding author): E-mail: zjhan@tju.edu.cn}\\
\\
{\normalsize 1.\ School of Mathematics, Tianjin University, Tianjin 300350, P. R. China}\\
{\normalsize 2.\ Department of Basic Science, Tangshan University, Tangshan 063000, P. R. China} }
\date{}
\title{\textbf{\  Stability of Multi-dimensional  Nonlinear Piezoelectric Beam with Viscoelastic Infinite Memory}\thanks{%
This work was supported by the National Natural Science Foundation of China NSFC-61773277, 62073236.}} \maketitle

\begin{abstract}
The long time behavior of a kind of fully magnetic effected nonlinear multi-dimensional piezoelectric beam with viscoelastic infinite memory is considered. The well-posedness of this nonlinear coupled PDEs' system is showed by mean of the semigroup theories and Banach fixed point theorem.  Based on frequency domain analysis,  it is proved that the corresponding coupled linear system can be indirectly stabilized exponentially by only one viscoelastic  infinite memory term, which is located on one equation of these strongly coupled PDEs.  Then the exponential decay of the solution to the nonlinear coupled PDEs' system  is established by the energy estimation method under certain conditions.



\end{abstract}

\section{Introduction}

\noindent

{Let $\Omega\subset\mathbb{R}^n$ be a bounded domain with smooth boundary $\partial \Omega$ satisfying $\partial \Omega=\Gamma_0\cup \Gamma_1,\; \Gamma_0\cap \Gamma_1=\emptyset$, {and $\mathbf{n}$ be the
unit outward normal vector of $\Gamma_1$.}}

In the present work, we consider the fully dynamic magnetic effect on the model of  piezoelectric beam with only one  viscoelastic memory involved, whose dynamic behavior is described by elasticity equation and charge equation coupling via the piezoelectric constants, which is given as follows.
\begin{equation}\label{sys1}
\left\{
\begin{array}{ll}
\rho v_{tt}(x,t)=\alpha \Delta v(x,t)-\gamma\beta \Delta p(x,t)+f_1(v,p)-\int_{0}^{\infty}g(s)\Delta v(x,t-s)ds,&x\in \Omega,~t>0,\\
\mu p_{tt}(x,t)=\beta \Delta p(x,t)-\gamma\beta \Delta v(x,t)+f_2(v,p),&x\in \Omega,~t>0,\\
v(x,t)=p(x,t)=0,&x\in \Gamma_0,~t>0,\\
\alpha \frac{\partial v}{\partial \mathbf{n}}(x,t)-\gamma\beta \frac{\partial p}{\partial \mathbf{n}}(x,t)=\beta \frac{\partial p}{\partial \mathbf{n}}(x,t)-\gamma\beta \frac{\partial v}{\partial \mathbf{n}}(x,t)=0,&x\in \Gamma_1,~t>0,\\
v(x,0)=v_0(x),~~v_t(x,0)=v_1(x),\; p(x,0)=p_0(x),~p_{t}(x,0)=p_{1}(x),~&x\in\Omega,\\
v(x,-s)=h(x,s),~&x\in \Omega,~s>0,
\end{array}
\right.
\end{equation}
where  $v(x,t)$ and $p(x,t)$ are denoted by the transverse displacement of the beam and the total load of
the electric displacement along the transverse direction at each point $x\in\Omega$, respectively. $v_0, v_1, p_0, p_1$ are the given initial data.  The  coefficients $\rho, \alpha, \gamma, \mu, \beta>0$ are the
mass density per unit volume, elastic stiffness, piezoelectric coefficient, magnetic permeability, impermeability
coefficient of the beam, respectively and satisfy $\alpha> \gamma^2\beta$, and thus, there always exists a positive constant  $\alpha_1$ such that $\alpha_1=\alpha-\gamma^2\beta$. The functions $f_i(v,p),\; i=1,2$ and $h(x,s)$ are nonlinear source terms and memory history function, respectively. $g(s)$ is the memory kernel function (also named as  ``relaxation function"),
{and the following assumptions on $g(s)$ are imposed:

(A1)  $g\in L^1(\mathbb{R}_+)\cap H^1(\mathbb{R_+})$ satisfies {$0<\zeta:=\int_{0}^{\infty}g(s)ds<\alpha_1$} and $g(s)>0$ for $s\in \mathbb{R}_+$;

(A2) For any $s\in \mathbb{R}_+$, $g'(s)<0$ and there exist two constants $k_0>0$ and $k_1>0$ such that {$-k_0g(s)\leq g'(s)\leq-k_1g(s)$.}

The piezoelectric material, as a kind of multi-functional smart material, is tremendously applied in industrial fields in recent years, such as ultrasonic welders, micro-sensors, inchworm robots,  wearable human-machine interfaces  and so on. For more details, we refer  to the typical works \cite{cd,ma1,ma2,anderson1,anderson2,smith} and the references therein. It is worth mentioning that based on Hamilton's principle, Morris and \"{O}zer in \cite{ma1} first proposed the theory of piezoelectric beam with  fully magnetic effect and derived a kind of PDEs' dynamical model:
 \begin{equation}\label{morri}
\left\{
\begin{array}{ll}
\rho v_{tt}(x,t)=\alpha v_{xx}(x,t)-\gamma\beta p_{xx}(x,t),&x\in (0,L),~t>0,\\
\mu p_{tt}(x,t)=\beta p_{xx}(x,t)-\gamma\beta v_{xx}(x,t),&x\in (0,L),~t>0,\\
v(0,t)=p(0,t)=0,&t>0,\\
\alpha v_x(L,t)-\gamma\beta p_x(L,t)=0,&t>0,\\
\beta p_x(L,t)-\gamma \beta v_x(L,t)=-\frac{V(t)}{h},&t>0,
\end{array}
\right.
\end{equation}
 where $V(t)$ denotes the voltage applied at the electrodes. This proposed model aims to compensate for  the deficiency of the traditional piezoelectric beam theory dominated by Maxwell's equation which ignores the dynamic interactions of electro-magnetism, and the electromagnetic effect must be taken into account in some situations, as mentioned in reference \cite{yang}. Recently, based on its abundant applications in industrial fields, the long time behavior  of the piezoelectric beam system with various damping mechanisms has been attracted by many scholars. In \cite{ma2, ma3}, Morris and \"{O}zer proved that the system can be strongly stabilized, but not exponentially stable if only one frictional damping {is} actuated on either of these two boundaries. However,  it was proved in \cite{ma1} that the system can achieve exponential stability if the frictional {damping}  are actuated on both boundaries simultaneously.} Compared to boundary damping, \cite{am4} showed that only one internal frictional damping  acting on either of these two equations is sufficient to stabilize this coupled system exponentially. Different from the above mentioned literature, we replace the frictional damping by a viscoelastic  infinite memory term $\int_{0}^{\infty}g(s)\Delta v(x,t-s)ds$, {which is presented only in the first equation of the strongly coupled system \eqref{sys1}.  This type of memory term may also exist in thermoelastic materials in low temperature, the stability of which was discussed by  \cite{wangjunmin}}. The objective of this work is to identify to which extent,  this multi-dimensional coupled PDEs system can be ``indirectly" stabilized through a viscoelastic memory term acting only in one of these two equations.

Besides, it is a known fact that the nonlinear models can describe the natural phenomena more accurately. As a great number of extra uncertainties often occur in practical engineering, many nonlinear elastic systems (vibrating strings, beams, or others)  are proposed and aroused mathematicians and engineers' interest. For instance, Rivera et al. in \cite{Rivera} studied a nonlinear {\color{blue}system} of Timoshenko type in a one-dimensional bounded domain and investigated the exponential stability of the system.  Mustafa et al. in \cite{mus} {obtained} an explicit and general decay result of a class of nonlinear Timoshenko beam system. \cite{Kh} studied the decay rates of a one-dimensional porous-elastic beam  with infinite memory under a nonlinear damping mechanism. \cite{freitas}  studied the long time dynamics of a kind of nonlinear piezoelectric {\color{blue}beam} with fractional damping and thermal effects, and \cite{fre}  considered a nonlinear piezoelectric system with delay effect,  in which the global attractor and exponential {attractor} are studied. For other kinds of PDEs on {\color{blue}this} issue, we refer to \cite{pao} for a semilinear wave {\color{blue}equation} with viscoelastic damping and delay feedback and \cite{ch2} for two nonlinear systems including Korteweg-de Vries-Burgers and Kuramoto-Sivashinsky equations with memory, and the references therein.

In this paper, we consider the stability of a nonlinear piezoelectric beam system \eqref{sys1} with only one viscoelastic memory involved. The novelties and  contributions are mainly in twofold:

{1. Through the in-depth research, we found that there is no work on the long time behavior of the multi-dimensional nonlinear piezoelectric beam system with infinite memory terms.} Different from the existing literature on damped piezoelectric beam systems, we consider the damping mechanism by a viscoelastic memory term $\int_{0}^{\infty}g(s)\Delta v(x,t-s)ds$. Especially,  it actuates only on one equation of the coupled PDEs \eqref{sys1} and the other equation in this system is ``indirectly" stabilized through the coupling.  We first show that the corresponding linear coupled PDEs system can be indirectly stabilized exponentially by the viscoelastic infinite memory. Especially, this decay rate is irrelevant to the relationships of wave speeds $\frac{\rho}{\alpha}$ and $\frac{\mu}{\beta}$,  which is totally different from the well-known Timoshenko beam (see \cite{Ammar}, \cite{alm}). Moreover,  the exponential decay of the nonlinear system \eqref{sys1} is further proved by a careful energy estimate.

{2. For the previous works on the stability analysis of the piezoelectric beams, the energy multiplier method is used exclusively. However, the infinite memory term living in such a model leads to the difficulty in the construction of energy multipliers. In order to solve this problem, the frequency domain method is  adopted in this work  to prove the exponential decay of the piezoelectric  beam system.}  To the best of the authors' knowledge, this could be the first work that the frequency domain method is employed to discuss the stability of the multi-dimensional piezoelectric beam system with viscoelastic memory term.
Some novel frequency multipliers {\color{blue}are} developed to help us overcome the technical difficulties caused by the strong coupling characteristics and  the memory term.

The rest of this paper is organized as follows. In section 2, the preliminary assumptions and spaces are presented. In section 3, the generation of semigroup $\{e^{\mathcal{A}t}\}_{t\geq0}$ of the linear part of the system \eqref{sys1} is discussed by the semigroup theories and the well-posedness of the nonlinear system  is further dealt with by the Banach fixed point theory. Section 4 is devoted to discussing the exponential stability of  system \eqref{sys1}. Finally, a concluding remark is given in section 5.


\section{Preliminaries and problem setting}
\qquad
This section is devoted to considering system (\ref{sys1}) in an appropriate Hilbert space setting. According to the approach of Dafermos \cite{dafermos}, let us first define a new variable  $\eta(x,t,s)$ for system (\ref{sys1})
\begin{equation}
\eta(x,t,s)=v(x,t)-v(x,t-s),~x\in\Omega,~t,s>0.
\end{equation}

It is easy to verify that $\eta_t(x,t,s)=v_t(x,t)-\eta_s(x,t,s)$ and
\begin{equation}
\label{eta2}\eta(x,t,0)=0,~x\in\Omega,~t>0.
\end{equation}

 Thus, system (\ref{sys1}) can be  rewritten as follows:
\begin{eqnarray}\label{sys2}
\left\{
\begin{array}{ll}
\rho v_{tt}(x,t)=(\alpha-\zeta) \Delta v(x,t)-\gamma\beta \Delta p(x,t)+f_1(v,p)+\int_{0}^{\infty}g(s)\Delta \eta(x,t,s)ds,&x\in \Omega,~t>0,\\
\mu p_{tt}(x,t)=\beta \Delta p(x,t)-\gamma\beta \Delta v(x,t)+f_2(v,p),&x\in \Omega,~t>0,\\
\eta_t(x,t,s)=v_t(x,t)-\eta_s(x,t,s),&x\in \Omega,~t,s>0,\\
v(x,t)=p(x,t)=0,&x\in \Gamma_0,~t>0,\\
\alpha \frac{\partial v}{\partial \mathbf{n}}(x,t)-\gamma\beta \frac{\partial p}{\partial \mathbf{n}}(x,t)=\beta \frac{\partial p}{\partial \mathbf{n}}(x,t)-\gamma\beta \frac{\partial v}{\partial \mathbf{n}}(x,t)=0,&x\in \Gamma_1,~t>0,\\
v(x,0)=v_0(x),~~v_t(x,0)=v_1(x),\;p(x,0)=p_0(x),~p_{t}(x,0)=p_{1}(x),~&x\in\Omega,\\
v(x,-s)=h(x,s),~&x\in \Omega,~s>0.\\
\end{array}
\right.
\end{eqnarray}
Define $H_{\Gamma_0}^{k}(\Omega)=\left\{f\in H^{k}(\Omega)\mid f=0~\text{on}~\Gamma_0\right\}$, where $H^{k}(\Omega)$ is $k$-order Sobolev space.
{We define the space $\Xi$ by
\[
\Xi:=\left\{\eta(x,s)\biggm|\begin{array}{c}
                                                                \eta(x,s)\in H_{\Gamma_0}^1(\Omega),\eta_s(x,s)\in H_{\Gamma_0}^1(\Omega) \\
                                                                \int_{0}^{\infty}\int_{\Omega}g(s)|\nabla\eta(x,s)|^2 dsdx<\infty
                                                                \end{array}\right\}
\]}
equipped with the inner product
\[
{( \eta,\tilde{\eta})_{\Xi}}=\int_{0}^{\infty}\int_{\Omega}g(s)\nabla \eta(x,s)\cdot\nabla\overline{\tilde{\eta}}(x,s)dxds.
\]

 Choose the state space \begin{equation}\label{204}
\mathcal{H}:=H_{\Gamma_0}^{1}(\Omega)\times L^2(\Omega)\times H_{\Gamma_0}^{1}(\Omega)\times L^2(\Omega) \times\Xi\end{equation}
with
\begin{eqnarray}\label{inner-product}
&&\left\langle(v,u,p,q,\eta)^T,(\tilde{v},\tilde{u},\tilde{p},\tilde{q},\tilde{\eta})^T\right\rangle_{\mathcal{H}}\cr
&=&\int_{\Omega}(\alpha_1-\zeta)\nabla v\cdot\nabla\overline{\tilde{v}}dx+\int_{\Omega}\rho u\overline{\tilde{u}}dx+\int_{\Omega}\beta(\gamma \nabla v-\nabla p)\cdot(\gamma\nabla\overline{\tilde{v}}-\nabla\overline{\tilde{p}})dx+\int_{\Omega}\mu q\overline{\tilde{q}}dx\cr
&&+\int_{0}^{\infty}\int_{\Omega}g(s)\nabla\eta(x,s)\cdot\nabla\overline{\tilde{\eta}}(x,s)dxds.
\end{eqnarray}
Under the assumptions (A1) and (A2),  it is easy to check that $(\mathcal{H}, \langle\cdot,\cdot\rangle_{\mathcal{H}})$ is a Hilbert space.

Define the system operator $\mathcal{A}$ in $\mathcal{H}$ by
\begin{equation}
\mathcal{A}\left(
   \begin{array}{c}
     v(x) \\
     u(x)\\
     p(x) \\
     q(x)\\
     \eta(x,s)
   \end{array}
 \right)=\left(
           \begin{array}{c}
             u(x) \\
             \frac{1}{\rho}[(\alpha-\zeta) \Delta v(x)-\gamma\beta \Delta p(x)+\int_{0}^{\infty}g(s)\Delta \eta(x,s)ds]\\
             q(x) \\
             \frac{1}{\mu}(\beta \Delta p(x)-\gamma\beta \Delta v(x))\\
             u(x)-\eta_s(x,s)

           \end{array}
         \right)
\end{equation}
with domain
{
\begin{equation}\label{domain1}
\mathcal{D}(\mathcal{A})=\left\{(v,u,p,q,\eta)^T\in \big((H^{2}(\Omega)\cap H_{\Gamma_0}^{1}(\Omega))\times H_{\Gamma_0}^{1}(\Omega)\big)^2\times\Xi \biggm|\begin{array}{c}
                                       \alpha \frac{\partial v}{\partial \mathbf{n}}-\gamma\beta \frac{\partial p}{\partial \mathbf{n}}=0\\
\beta \frac{\partial p}{\partial \mathbf{n}}-\gamma\beta \frac{\partial v}{\partial \mathbf{n}}=0\\
                                      \end{array}
~\text{on}~\Gamma_1\right\}.
\end{equation}
}

这种空间设置是为了用鲁磨菲利普斯定理，如果用预解族，空间设置就会是上面取消的那种情况

The source terms are described by a nonlinear function $\mathcal{F}:\mathcal{H}\rightarrow \mathcal{H}$ defined as
\begin{equation}\label{nonlinear203}
\begin{array}{c}
\mathcal{F}(v,u,p,q,\eta)=\left(
                       0 ,
                       {\frac{1}{\rho}f_1(v,p)},
                      0,
                       {\frac{1}{\mu}f_2(v,p)},0\right)^{T}.
\end{array}
\end{equation}

Set
$
X(t)=\left(
         v(x,t),
         v_t(x,t),
         p(x,t),
         p_t(x,t),
         \eta(x,s,t)\right)^{T} \;\mbox{and}\; X_0=\left(
         v_0,
         v_1,
         p_0,
         p_1,
        \eta_0
     \right)^{T}.
$ With the above notations, problem (\ref{sys2}) can be reformulated into the following abstract evolution equation:
\begin{equation}\label{eq2.8}
\left\{
\begin{array}{ll}
\frac{dX(t)}{dt}=\mathcal{A}X(t)+\mathcal{F}(X(t)),&t>0,\\
X(0)=X_0.
\end{array}
\right.
\end{equation}


\section{Well-posedness}

\subsection{The generation of $C_0$-semigroup $e^{\mathcal{A}t}$}

\begin{theorem}\label{th3.1}
Assume that the assumptions (A1)-(A2) hold. Then 
%
%
the operator $\mathcal{A}$ generates a $C_0$ semigroup of contractions $e^{\mathcal{A}t}$ on $\mathcal{H}$.
\end{theorem}

\begin{proof}
{First, we show that $\mathcal{A}$ is dissipative in $\mathcal{H}$. In fact, for any $W=(v,u,p,q,\eta)^T\in \mathcal{D}(\mathcal{A})$, a direct calculation yields that
\begin{eqnarray}\label{E}
\Re\langle\mathcal{A}W,W\rangle}_{\mathcal{H}&=&\Re\big(\int_{\Omega}(\alpha_1\!-\!\zeta)\nabla u \cdot\nabla \overline{v}dx\!+\!\int_{\Omega}\rho[\frac{1}{\rho}((\alpha\!-\!\zeta) \Delta v-\gamma\beta \Delta p\!+\!\int_{0}^{\infty}g(s)\Delta \eta(x,s)ds)]\overline{u}dx\cr
&&+\int_{\Omega}\beta(\gamma \nabla u-\nabla q)\cdot(\gamma \nabla \overline{v}-\nabla \overline{p})dx+\int_{\Omega}\mu(\frac{1}{\mu}(\beta \Delta p-\gamma\beta \Delta v))\overline{q}dx\cr
&&+\int_{0}^{\infty}\int_{\Omega}g(s)\nabla(u-\eta_s(x,s))\cdot\nabla\overline{\eta}(x,s)dxds\big)\cr
&=&\frac{1}{2}\int_{0}^{\infty}\int_{\Omega}g'(s)|\nabla \eta|^2dxds\leq 0,
\end{eqnarray}
which implies that $\mathcal{A}$ is dissipative.
}

Second, we show  $0\in \varrho(\mathcal{A})$ (where $\varrho(\mathcal{A})$ denotes the resolvent point set of $\mathcal{A}$). Indeed,
for any $(\xi_1,\xi_2,z_1,z_2,\nu)\in \mathcal{H}$, let us discuss the solvability of the equation
\begin{align}
 \label{resolvent}\mathcal{A}(v,u,p,q,\eta)^T=(\xi_1,\xi_2,z_1,z_2,\nu)^T,~~\text{for}~(v,u,p,q,\eta)^T\in \mathcal{D}(\mathcal{A}),
\end{align}
that is,
\begin{align}
 \label{w-p1}u= \xi_1,\\
 \label{w-p2}             \frac{1}{\rho}[(\alpha-\zeta) \Delta v-\gamma\beta \Delta p+\int_{0}^{\infty}g(s)\Delta \eta(x,s)ds]=\xi_2,\\
\label{w-p3}            q=z_1, \\
\label{w-p4}            \frac{1}{\mu}(\beta \Delta p-\gamma\beta \Delta v)=z_2,\\
\label{w-p5}            u-\eta_s=\nu.
\end{align}

Solving  $\eqref{w-p5}$, along with $\eqref{eta2}$ and $\eqref{w-p1}$, we have
\begin{align}
\label{w-solu}\eta=\int_{0}^{s}[\xi_1-\nu(r)]dr.
\end{align}

Substituting $\alpha_1=\alpha-\gamma^2\beta$ into $\eqref{w-p2}$, together with the assumption (A2) and $\eqref{w-p4}$, we transform $\eqref{w-p1}-\eqref{w-p5}$ into the following ones:
\begin{align}
\label{w-p2-1} &\Delta v=\frac{1}{\alpha_1-{\color{blue}\zeta}}[\gamma\mu z_2+\rho \xi_2-\int_{0}^{\infty}g(s)\Delta \eta(x,s)ds],\\
\label{w-p4-1} & \Delta p=\frac{1}{\alpha_1-{\color{blue}\zeta}}[\frac{(\alpha-{\color{blue}\zeta})\mu}{\beta}z_2+\rho \gamma \xi_2-\gamma\int_{0}^{\infty}g(s)\Delta \eta(x,s)ds].
\end{align}

Let $(\varphi,\psi)\in H_{\Gamma_0}^1(\Omega)\times H_{\Gamma_0}^1(\Omega)$. Multiplying $\eqref{w-p2-1}$ and $\eqref{w-p4-1}$ by $\overline{\varphi}$ and $\overline{\psi}$ and taking $L^2-$ inner product respectively, then integrating by parts yields that
\begin{align}
\nonumber&\int_{\Omega}\nabla v\cdot \nabla \overline{\varphi}dx+\int_{\Omega}\nabla p\cdot \nabla \overline{\psi}dx\\
\nonumber=&-\int_{\Omega}\frac{1}{\alpha_1-{\color{blue}\zeta}}[\gamma\mu z_2+\rho \xi_2-\int_{0}^{\infty}g(s)\Delta \eta(x,s)ds]\overline{\varphi}dx\\
\label{variation}&-\int_{\Omega}\frac{1}{\alpha_1-{\color{blue}\zeta}}[\frac{(\alpha-{\color{blue}\zeta})\mu}{\beta}z_2+\rho \gamma \xi_2-\gamma\int_{0}^{\infty}g(s)\Delta \eta(x,s)ds]\overline{\psi}dx.
\end{align}
In order to study the variation equation $\eqref{variation}$, we introduce the space $V^1(\Omega):=H_{\Gamma_0}^1(\Omega)\times H_{\Gamma_0}^1(\Omega)$ equipped with the norm
\[
\|(f,\widetilde{f})\|_{V^1}=(\int_{\Omega}(|\nabla f|^2+|\nabla \widetilde{f}|^2)dx)^{\frac{1}{2}}.
\]
Define the bilinear functional $B$ by
\begin{align}
\label{bilinear} &B((v,p),(\varphi,\psi))=\int_{\Omega}\nabla v\cdot \nabla \overline{\varphi}dx+\int_{\Omega}\nabla p\cdot \nabla \overline{\psi}dx.
\end{align}
It is obvious that $B$ is bounded and coercive in $V^1(\Omega)\times V^1(\Omega)$.

Introduce the functional $F$ defined on $V^1(\Omega)$ as
\begin{align}
 \nonumber F(\varphi,\psi)&=-\int_{\Omega}\frac{1}{\alpha_1-{\color{blue}\zeta}}[\gamma\mu z_2+\rho \xi_2-\int_{0}^{\infty}g(s)\Delta \eta(x,t,s)ds]\overline{\varphi}dx\\
\label{functional} &-\int_{\Omega}\frac{1}{\alpha_1-{\color{blue}\zeta}}[\frac{(\alpha-{\color{blue}\zeta})\mu}{\beta}z_2+\rho \gamma \xi_2-\gamma\int_{0}^{\infty}g(s)\Delta \eta(x,t,s)ds]\overline{\psi}dx.
\end{align}

By H\"{o}lder inequality and Poincar\'{e} inequality, simple calculation shows that $F$ is bounded in $V^1(\Omega)$. Then it follows by Lax-Millgram theorem that there is a unique solution $(v,p)\in V^1(\Omega)$ of $\eqref{variation}$. According to the result in \cite[Chapter 3, p.92]{gwess}{\color{blue},  the equation} $\eqref{resolvent}$ admits a unique solution $(v,u,p,q,\eta)^T\in \mathcal{D}(\mathcal{A})$ and hence $0\in \varrho(\mathcal{A})$.  Therefore, the well-known Lumer-Phillips theorem \cite{brez,pazy} shows that the desired result follows.
\end{proof}

{Note that the generation of $C_0$-semigroup implies the homogeneous problem (the linear part of (\ref{eq2.8})) is well-posed.}  In the following subsection, we discuss the well-posedness of the nonlinear problem (\ref{eq2.8}).

\subsection{Well-posedness of nonlinear system}
\qquad
{{We now prove a global existence result for problem \eqref{eq2.8} with sufficiently small initial conditions.} In order to ensure the solvability of the nonlinear problem \eqref{sys1}, {we assume the following hypotheses on the nonlinear source terms $f_i(v,p),i=1,2$:

(H1) $f_i(v,p)$ are locally Lipschitz continuous;

(H2) $f_i(v,p)$ fulfill the growth condition: there exist $r\geq0$ and $c>0$ such that
\begin{align}\label{li-gro}
| f_1(v,p)|\bigvee| f_2(v,p)|\leq c(|v|+|p|)(|v|^{r}+|p|^{r}+1),
\end{align}
where $|f_1|\bigvee |f_2|$ denotes the maximum of $|f_1|$ and $|f_2|$.}

Inspired by the work \cite{freitas}, we  obtain the following lemma.
\begin{lemma}\label{lip}用不用局部？
Suppose that the assumption {(H1)} holds, then $\mathcal{F}:\mathcal{H}\rightarrow \mathcal{H}$ is {locally} Lipschitz continuous.
\end{lemma}}

Thanks to Theorem \ref{th3.1}, we know that $\mathcal{A}$ generates a $C_0$ semigroup $e^{\mathcal{A}t}$ on $\mathbb{X}$, and  the problem (\ref{eq2.8}) can be expressed in the form of an equivalent functional integral equation
\begin{equation}\label{sol}
X(t)=
   e^{\mathcal{A}t}X_0+\int_{0}^{t}e^{\mathcal{A}(t-s)}\mathcal{F}(X(s))ds, ~t\geq0. \\
\end{equation}

Define $C(I,\mathcal{H})$ be the space consisting of the $\mathcal{H}$-value continuous functions on the compact set $I\subset \mathbb{R}_+$ and equipped with the classical norm $\|x\|_{\infty}=\max\limits_{t\in I}\|x(t)\|_{\mathcal{H}}$, then $C(I,\mathcal{H})$ is a Banach space.
By means of Banach fixed point theorem, we will show the existence and uniqueness of the solution of the integral equation (\ref{sol}).

{\begin{theorem}\label{th-sol}
Let $X_0$ and $\mathcal{F}$ be defined as in \eqref{eq2.8}. Assume that the {\color{blue}hypotheses (H1) and (H2) hold}. Then for any initial value $X_0\in \mathcal{H}$ {with $\|X_0\|_{\mathcal{H}}$ sufficiently small,} the integral equation \eqref{sol} has a unique global solution on $C(\mathbb{R}_+,\mathcal{H})$.
\end{theorem}
\begin{proof}
It can be easily verified that $X(\cdot)$ is a $\mathcal{H}$-value continuous function with respect to variable $t$. For $a>0$ fixed, we define an operator $\mathcal{J}$ on $C ([0,a],\mathcal{H})$ by
\[
\mathcal{J}X(t)=e^{\mathcal{A}t}X_0+\int_{0}^{t}e^{\mathcal{A}(t-s)}\mathcal{F}(X(s))ds,~~\forall t\in[0,a].
\]
Note that the solvability of (\ref{sol}) is equivalent to the fixed point of the operator equation $\mathcal{J}X = X$.

Since $e^{\mathcal{A}t}$ is a $C_0$-semigroup, there exist $M>0$ and $\omega>0$ such that $\|e^{\mathcal{A}t}\|\leq M e^{wt},~t\geq \mathbb{R}_+$.
In view of the {\color{blue}hypothesis} (H1), we know that Lemma \ref{lip} holds. Denote $B_K=\{\varphi \in C([0,a],\mathcal{H})| \|\varphi\|_{\mathcal{H}}\leq K\}$. For any given $X,\widetilde{X}\in B_K$, it follows from Lemma \ref{lip} that
$$
\|\mathcal{F}(X)-\mathcal{F}(\tilde{X})\|_{\mathcal{H}}\leq L(K)\|X-\tilde{X}\|_{\mathcal{H}},
$$
where $L(\cdot)$ is the Lipschitz coefficient. Then $\mathcal{F}$ is {\color{blue}uniformly bound} on $B_K$, and hence we denote $N=\sup\limits_{X\in B_K}\|\mathcal{F}(X)\|_{\mathcal{H}}$. {Choose a sufficiently small constant $b>0$ such that $b\leq a$, combining with the assumption for the small initial value, it deduces that $(M L (K))^2 M_bb<1$ and $Me^{\omega b}(\|X_0\|_{\mathcal{H}}+Nb/\omega)\leq K$, where $M_b=\frac{1}{2\omega}(e^{2\omega b}-1)$.} Let $D=C([0,b],B_K)$, which is a closed sphere of radius $K$ in $C ([0,b],\mathcal{H})$. It is easy to check that $\mathcal{J}:D\rightarrow D$. For $X,Y\in D$, making use of H\"{o}lder inequality, we have
\begin{eqnarray*}
\|\mathcal{J}X(s)-\mathcal{J}Y(s)\|_{\mathcal{H}}^2&\leq&\big(\int_{0}^{t}\big\|e^{\mathcal{A}(t-s)}[\mathcal{F}(X(s))-\mathcal{F}(Y(s))]\big\|_{\mathcal{H}}ds\big)^2\\
&\leq & M^2\big(\int_{0}^{t}e^{\omega(t-s)}\|\mathcal{F}(X(s))-\mathcal{F}(Y(s))\|_{\mathcal{H}}ds\big)^2\\
&\leq&M^2M_b (L(K))^2\int_{0}^{t}\|X(s)-Y(s)\|_{\mathcal{H}}^2ds\\
&\leq&{(M L (K))^2 M_bb\|X-Y\|_{\infty}^2,}
\end{eqnarray*}
which shows that $\mathcal{J}$ is a contraction operator. By Banach fixed point theorem, there is a unique solution $X(t)\in C([0,b],\mathcal{H})$. This proves the local existence and uniqueness of the solution to system \eqref{eq2.8}.

Next, we
claim the global existence of the solution. {Similar to} \cite{sn1}, in order to prove the global
existence, we {\color{blue}shall} derive a priori estimate of this solution. Let $X(t)$ be a solution of \eqref{sol} on the interval $[0,T]$, $T>0$, we first estimate the nonlinear term. {\color{blue}Using} the hypothesis {\color{blue}(H2)} and Poincar\'{e} inequality, it follows that
\begin{eqnarray}\nonumber
\|\mathcal{F}(X)\|_{\mathcal{H}}^2&=&\frac{1}{\rho}\int_{\Omega}|f_1(v,p)|^2dx+\frac{1}{\mu}\int_{\Omega}|f_2(v,p)|^2dx\cr
&\leq&(\frac{1}{\rho}+\frac{1}{\mu})c^2\int_{\Omega}(|v|+|p|)^2(|v|^r+|p|^r+1)^2 dx\cr
&\leq&2(\frac{1}{\rho}+\frac{1}{\mu})c^2l_0\int_{\Omega}(|\nabla v|^2+|\nabla p|^2) dx\cr
&\leq&2(\frac{1}{\rho}+\frac{1}{\mu})c^2l_0\delta\int_{\Omega}(\alpha_1-\zeta)|\nabla v|^2+\beta|\gamma\nabla v-\nabla p|^2dx\cr
&\leq&\vartheta\|X\|_{\mathcal{H}}^2,
\end{eqnarray}
where $\delta=\max\{\frac{\gamma^2+1}{\alpha_1-\zeta},\frac{1}{\beta}\}$ and $\vartheta=2(\frac{1}{\rho}+\frac{1}{\mu})c^2l_0\delta$, {in which $c$ is the same as in (\ref{li-gro}) and $l_0>0$ is related to $r$ and Poincar\'{e} constant.} Thus
\begin{eqnarray}\label{e-a1}
\|X(t)\|_{\mathcal{H}}^2\leq M^2e^{2\omega T}\|X_0\|_{\mathcal{H}}^2+M^2 M_T\vartheta \int_{0}^{t}\|X(s)\|_{\mathcal{H}}^2ds,
\end{eqnarray}
where $M_T=\frac{1}{2\omega}(e^{2\omega T}-1)$.
Applying Gronwall's inequality to \eqref{e-a1}, it follows that
\begin{eqnarray}\nonumber
{\|X(t)\|_{\mathcal{H}}\leq e^{\frac{(M\sqrt{M_T\vartheta})^2 T}{2}}Me^{\omega T}\|X_0\|_{\mathcal{H}},~t\in[0,T].}
\end{eqnarray}
Since $T>0$ is arbitrary chosen and can be large enough, the global existence of the solution is verified.
\end{proof}
}

%

\section{Stability analysis}
\qquad
{This section is devoted to discussing the long time behavior of the solution to the  nonlinear  system (\ref{sys1}). To do this, let us first consider the stability of  the $C_0$- semigroup  associated to the linear part of this system.
\subsection{Exponential stability of $C_0$-semigroup $e^{\mathcal{A}t}$}
\qquad
 Based on the  frequency domain method, we obtain the following result.
\begin{theorem}\label{t-4-1}
	Assume that the assumptions (A1)-(A2) hold. Then the semigroup $e^{\mathcal{A}t}$ associated with the linear part of  system (\ref{sys2}) is exponentially stable on $\mathcal{H}$, that is, there are two constants $\widetilde{M}>0$ and {\color{blue}$\widetilde{r}>0$} such that the semigroup  $e^{\mathcal{A}t}$ satisfies
	\begin{equation}\label{2.5}
	\|e^{\mathcal{A}t}(v_{0},v_{1},p_{0}, p_{1}, \eta_0)\|_{\mathcal{H}}\leq{\widetilde{M}}{\color{blue}e^{-\widetilde{r}t}}\|(v_{0},v_{1},p_{0}, p_{1}, \eta_0)\|_{\mathcal{H}}.
	\end{equation}
\end{theorem}


In order to show the above theorem, let us recall the following frequency characteristics on the exponential stability  of $C_0$-semigroups of contractions on Hilbert spaces {
(see \cite{Gear, huangfalun, pruss}).}
\begin{lemma}\label{l-4-1}
Let $\mathcal{A}$ be the infinitesimal generator of  a bounded $C_0$ semigroup $e^{\mathcal{A}t}$ on a Hilbert space $\mathcal{H}$ and {satisfy} $i\mathbb{R}\subset\varrho(\mathcal{A})$. Then $e^{\mathcal{A}t}$ is exponentially stable if and only if the following condition holds:
\begin{equation}\label{403}
 \sup_\lambda\limits\{\|(i\lambda-\mathcal{A})^{-1}\|_{\mathcal{H}}\mid\lambda\in \mathbb{R}\}<\infty.
 \end{equation}
\end{lemma}



\noindent{\bf Proof of Theorem \ref{t-4-1}.}\,\,
 By Lemma \ref{l-4-1}, it is sufficient to verify that the conditions $i\mathbb{R}\subset\varrho(\mathcal{A})$ and \eqref{403} hold. For clarity, the proof is divided by two parts.
	
	\vspace{0.5cm}
\noindent{\bf Part I.
	We show  (\ref{403})   holds.}

By contradiction,
if it is not true, then thanks to Banach-Steinhaus theorem, there exist  sequences $X_n=(v_n,u_n,p_n,q_n,w_n)\in \mathcal{D}(\mathcal{A})$ with $\|X_n\|_{\mathcal{H}}=1$ and $\lambda_n\rightarrow \infty$ such that
\begin{equation}\label{bs1}
(i\lambda_n-\mathcal{A})X_n\equiv (\xi^1_n,\xi^2_n,z^1_n,z^2_n,\nu_n)\rightarrow 0~\text{in}~\mathcal{H},
\end{equation}
specifically,
\begin{align}
\label{e1} i \lambda_n v_n-u_n= \xi^1_n&\rightarrow 0,~\text{in}~H^1_{\Gamma_0}(\Omega),\\
\label{e2}i \lambda_nu_n-\frac{1}{\rho}[(\alpha-\zeta) \Delta v_{n}-\gamma\beta \Delta p_{n}+\int_{0}^{\infty}g(s)\Delta \eta_n(x,s)ds]=\xi^2_n&\rightarrow0,~\text{in}~L^2(\Omega),\\
\label{e3} i \lambda_n p_n-q_n=z^1_n &\rightarrow0,~\text{in}~H^1_{\Gamma_0}(\Omega),\\
 \label{e4}  i \lambda_n q_n- \frac{1}{\mu}(\beta \Delta p_{n}-\gamma\beta \Delta v_{n})=z^2_n&\rightarrow0,~\text{in}~L^2(\Omega),\\
 \label{e5}     i \lambda_n \eta_n-( u_n -\eta_{n,s})=\nu_n&\rightarrow0,~\text{in}~\Xi
\end{align}
with the boundary conditions
\begin{align}
\label{b1} v_n=p_n=0,~&\text{on}~ \Gamma_0,\\
\label{b2}\alpha \frac{\partial v_n}{\partial \mathbf{n}}-\gamma\beta \frac{\partial p_n}{\partial \mathbf{n}}=\beta \frac{\partial p_n}{\partial \mathbf{n}}-\gamma\beta \frac{\partial v_n}{\partial \mathbf{n}}=0,~&\text{on}~ \Gamma_1.
\end{align}
It is easy to check that $(\ref{b1})-(\ref{b2})$ can be simplified as Dirichlet-Neumann type:
\begin{align}
\label{d-n}
\begin{array}{l}
 v_n|_{\Gamma_0}=p_n|_{\Gamma_0}=\frac{\partial v_n}{\partial \mathbf{n}}|_{\Gamma_1}=\frac{\partial p_n}{\partial \mathbf{n}}|_{\Gamma_1}=0.
\end{array}
\end{align}

Note that $\eta_n,\eta_{n,s}, u_n=0$ on $\Gamma_0$, then \eqref{e5} implies that
\begin{align}
\label{e5-out}     i \lambda_n \eta_n-( u_n -\eta_{n,s})&\rightarrow0,~\text{in}~L_{g}^2(\mathbb{R}_+,L^2(\Omega)),
\end{align}
where the norm in $L_{g}^2(\mathbb{R}_+,L^2(\Omega))$ is defined as the conventional one, that is, $$\|\varphi\|_g= \int_{0}^{\infty}\int_{\Omega}g(s)|\varphi(x,s)|^2 dsdx.$$

In the sequel of this part, we aim to show
\begin{equation}\label{hh14}
|\eta_n\|_{\Xi},\; \|u_n\|_{L^2(\Omega)},\; \|\nabla v_{n}\|_{L^2(\Omega)},\; \|\gamma \nabla v_{n}-\nabla p_{n}\|_{L^2(\Omega)},\; \|q_n\|_{L^2(\Omega)}=o(1).
\end{equation}
As long as  \eqref{hh14} is proved, we can get directly that $\|X_n\|_{\mathcal{H}}=o(1)$ due to (\ref{inner-product}), which
contradicts the fact that $\|X_{n}\|_{\mathcal{H}}= 1$, and thus, the proof of {\bf Part I} can be finished.

 For this aim, we divide it by five steps.

\vspace{0.3cm}
\noindent{\it Step 1. $\|\eta_n\|_\Xi\rightarrow0,\quad n\to \infty.$}

By virtue of the dissipativeness of $\mathcal{A}$ (see  (\ref{E})) and $(\ref{bs1})$, we get
\begin{align}
\nonumber \int_{0}^{\infty}\int_{\Omega}g'(s)|\nabla\eta_{n}|^2dxds&\rightarrow0,~\text{in}~\mathbb{C}.
\end{align}

Due to the assumption (A2), we obtain that
\begin{eqnarray*}
0\leftarrow-\int_{0}^{\infty}\int_{\Omega}g'(s)|\nabla\eta_{n}|^2dxds\geq k_1\int_{0}^{\infty}\int_{\Omega}g(s)|\nabla\eta_{n}|^2dxds\geq 0,
\end{eqnarray*}
which implies that
\begin{align}
\label{w0} \|\eta_n\|_{\Xi}\rightarrow0.
\end{align}

\vspace{0.3cm}
\noindent{\it Step 2. $ \|u_n\|_{L^2(\Omega)}\rightarrow0,\quad n\to\infty.$}

Taking the $L^2-$inner product of $(\ref{e5-out})$ with $g(s)u_n(x)$ and then integrating it with respect to $s$, we have
\begin{align}
\label{e5e-g-u-fir} \int_{0}^{\infty}\int_{\Omega}i \lambda_n g(s)\overline{\eta_n}u_ndxds+ \int_{0}^{\infty}\int_{\Omega}g(s)\overline{\eta_{n,s}}u_ndxds-\int_{0}^{\infty}\int_{\Omega}g(s)|u_n|^2dxds\rightarrow0.
\end{align}
We now estimate each term in (\ref{e5e-g-u-fir}). {Recall $\zeta=\int_{0}^{\infty}g(s)ds>0$.}

{\it Observation I. }
In view of \eqref{e2}, using H\"{o}lder inequality, we have
{\begin{eqnarray}\nonumber
&&\bigg|\int_{0}^{\infty}\int_{\Omega}i \lambda_n g(s)\overline{\eta_n}u_ndxds\bigg|^2\cr
&=&\frac{1}{\rho^2}\bigg|\int_{0}^{\infty}\int_{\Omega}g(s)\overline{\eta_{n}}(x,s)[(\alpha-\zeta)\Delta v_{n}-\gamma\beta \Delta p_{n}+\int_{0}^{\infty}g(r)\Delta \eta_n(x,r)dr]dxds\bigg|^2\cr
&=&\frac{1}{\rho^2}\bigg|-(\alpha-\zeta)\int_{0}^{\infty}\int_{\Omega}g(s)\nabla\overline{\eta_n}\cdot \nabla v_ndxds+\gamma\beta\int_{0}^{\infty}\int_{\Omega}g(s)\nabla\overline{\eta_n}\cdot \nabla p_ndxds\cr
&&-\int_{0}^{\infty}\big(\int_{\Omega}g(s)\nabla \overline{\eta_n}(x,s)\cdot (\int_{0}^{\infty}g(r)\nabla \eta_n(x,r)dr) dx\big)ds\bigg|^2\cr
&\leq&\frac{3}{\rho^2}(\alpha-\zeta)^2\zeta\|\nabla v_n\|_{L^2(\Omega)}^2\int_{0}^{\infty}\int_{\Omega}g(s)|\nabla\eta_n|^2dxds\cr
&&+\frac{3\gamma^2\beta^2}{\rho^2}\zeta\|\nabla p_n\|_{L^2(\Omega)}^2\int_{0}^{\infty}\int_{\Omega}g(s)|\nabla\eta_n|^2dxds\cr
&&+\frac{3\zeta^2}{\rho^2}\big(\int_{0}^{\infty}\int_{\Omega}g(s)|\nabla \eta_n(x,s)|^2dxds\big)^2,
\end{eqnarray}
which along with $\|\eta_n\|_\Xi\rightarrow 0$ yields that the first term in (\ref{e5e-g-u-fir}) satisfies
\begin{align}
\nonumber\int_{0}^{\infty}\int_{\Omega}i \lambda_n g(s)\overline{\eta_n}u_ndxds\rightarrow 0.
\end{align}}

{\it Observation II.} Thanks to (A2) and \eqref{w0},  integrating by part and applying Cauchy-Schwartz inequality induce that
{\begin{eqnarray}\nonumber
\bigg|\int_{0}^{\infty}\int_{\Omega}g(s)\eta_{n,s}\overline{u_n}dxds\bigg|^2&=&\bigg|-\int_{0}^{\infty}\int_{\Omega}g'(s)\eta_{n}(x,s)\overline{u_n}(x)dxds\bigg|^2\cr
&\leq&k_0^2\bigg|\int_{0}^{\infty}\int_{\Omega}g(s)\eta_{n}(x,s)\overline{u_n}(x)dx ds\bigg|^2\cr
&\leq&k_0^2\bigg|\int_{0}^{\infty}\big(\int_{\Omega}|g(s)\eta_n(x,s)|^2 dx\big)^\frac{1}{2}\big(\int_{\Omega}|u_n(x)|^2 dx\big)^{\frac{1}{2}}ds\bigg|^2\cr
&\leq&k_0^2\|u_n\|_{L^2(\Omega)}^2\int_{0}^{\infty}g(s)ds\int_{0}^{\infty}\int_{\Omega}g(s)|\eta_n(x,s)|^2 dxds\cr
&\leq&k_0^2\|u_n\|_{L^2(\Omega)}^2\zeta\int_{0}^{\infty}\int_{\Omega}g(s)|\nabla\eta_n(x,s)|^2 dxds\rightarrow0.
\end{eqnarray}}

Substituting {\color{blue}{\it Observation I} and {\it Observation II}} into $(\ref{e5e-g-u-fir})$ yields that $\int_{0}^{\infty}\int_{\Omega}g(s)u_n\overline{u_n}dxds\rightarrow0$.  Note that
\begin{eqnarray}\nonumber
0\leftarrow\int_{0}^{\infty}\int_{\Omega}g(s)u_n\overline{u_n}dxds&=&\int_{0}^{\infty}g(s)ds\int_{\Omega}|u_n|^2 dx,
\end{eqnarray}
which together with the assumption  (A1) shows that
\begin{align}
\label{un} \|u_n\|_{L^2(\Omega)}\rightarrow0.
\end{align}

\vspace{0.3cm}
\noindent{\it Step 3. $\|\nabla v_{n}\|_{L^2(\Omega)}\to 0,\quad n\to \infty$.}

Substituting $(\ref{e1})$ and $(\ref{e3})$ into $(\ref{e2})$ and $(\ref{e4})$, respectively, we get
\begin{align}
\label{ne2} -\lambda^2_n\rho v_n-[(\alpha-\zeta) \Delta v_{n}-\gamma\beta \Delta p_{n}+\int_{0}^{\infty}g(s)\Delta \eta_n(x,s)ds]\rightarrow0,~\text{in}~L^2(\Omega),&\\
\label{ne4}  -\lambda^2_n \mu p_n-(\beta\Delta p_{n}-\gamma\beta \Delta v_{n})\rightarrow0,~\text{in}~L^2(\Omega).&
\end{align}

Note that thanks to $\alpha_1=\alpha-\gamma^2\beta>0$,   $(\ref{ne2})$ and $(\ref{ne4})$  can be transformed into the following form, respectively.
\begin{eqnarray}
&&-\lambda^2_n\rho v_n-(\alpha_1-\zeta) \Delta v_{n}-\gamma\beta(\gamma \Delta v_{n}- \Delta p_{n})
-\int_{0}^{\infty}g(s)\Delta \eta_n(x,s)ds\rightarrow0,~\text{in}~L^2(\Omega),\label{ne21}\\
&&-\lambda^2_n\gamma \mu p_n+\gamma\beta( \gamma\Delta v_{n}-\Delta p_{n})\rightarrow0,~\text{in}~L^2(\Omega).\label{ne41}
\end{eqnarray}

Thus, adding (\ref{ne21}) and (\ref{ne41}) so as to eliminate the common item $\gamma\beta(\gamma\Delta v_{n}-\Delta p_{n})$, we obtain
\begin{align}
\label{ne2ne4} &\lambda^2_n\rho v_n+\lambda^2_n\gamma \mu p_n+(\alpha_1-\zeta)\Delta v_{n}+\int_{0}^{\infty}g(s)\Delta \eta_n(x,s)ds\rightarrow0,~~\text{in}~L^2(\Omega).
\end{align}

Then, taking the  $L^2-$inner product of $(\ref{ne2ne4})$ with $v_n$, along with  the boundary conditions $(\ref{d-n})$, we have
\begin{equation}
\label{ne2ne4-1}\int_{\Omega}\lambda^2_n\rho |v_n|^2dx+\int_{\Omega}\lambda^2_n\gamma \mu p_n\overline{v_n}dx-\int_{\Omega}(\alpha_1-\zeta) |\nabla v_{n}|^2dx-\int_{0}^{\infty}\int_{\Omega}g(s)\nabla \eta_n\cdot \nabla \overline{v_n}dxds\rightarrow0.
\end{equation}

By $(\ref{e1})$, $(\ref{e3})$ and Cauchy-Schwartz inequality, we get that the second term in (\ref{ne2ne4-1}) satisfies
\begin{align}
\label{lam-p-v}\bigg|\int_{\Omega}\lambda^2_n\gamma \mu p_n\overline{v_n}dx\bigg|^2\sim\bigg|\gamma \mu\int_{\Omega}q_n\overline{u_n}dx\bigg|^2\leq(\gamma\mu)^2\int_{\Omega}|u_n|^2dx\int_{\Omega}|q_n|^2dx\rightarrow 0.
\end{align}
{Here} we have used (\ref{un}) and the boundedness of $\|q_n\|_{L^2({\color{blue}\Omega})}$.

By Cauchy-Schwartz inequality, we have
{\begin{align}
\label{eta-v}\bigg|-\int_{0}^{\infty}\int_{\Omega}g(s)\nabla \eta_n\cdot \nabla \overline{v_n}dxds\bigg|^2
&\leq\bigg|\int_{0}^{\infty}\big(\int_{\Omega}|g(s)\nabla \eta_n|^2 dx\big)^{\frac{1}{2}} \big(\int_{\Omega}|\nabla v_n|^2dx\big)^{\frac{1}{2}}ds\bigg|^2\cr
&\leq \|\nabla v_{n}\|_{L^2(\Omega)}^2\zeta\int_{0}^{\infty}\int_{\Omega}g(s)|\nabla \eta_n|^2dxds\cr
&\to 0.
\end{align}}
This together with \eqref{w0}, $(\ref{ne2ne4-1})$ and $(\ref{lam-p-v})$ shows that
\begin{align}
\nonumber  &\int_{\Omega}\lambda^2_n\rho |v_n|^2dx-\int_{\Omega}(\alpha_1-\zeta) |\nabla v_{n}|^2dx\rightarrow0,
\end{align}
which along with (\ref{e1}) and (\ref{un}) leads to
\begin{align}
\label{vn} &\|\nabla v_{n}\|_{L^2(\Omega)}\rightarrow0.
\end{align}

\vspace{0.3cm}
\noindent{\it Step 4. $\|q_n\|_{L^2(\Omega)}\to 0,\quad n\to \infty.$}
\vspace{0.2cm}

Taking the $L^2-$ inner product of (\ref{ne2ne4}) with $p_n$ and integrating it by parts, along with the boundary conditions (\ref{d-n}), we obtain
\begin{equation}\label{434}
\int_{\Omega}\lambda^2_n\rho v_n \overline{p_n}dx+\int_{\Omega}\lambda^2_n\gamma \mu |p_n|^2dx-\int_{\Omega}(\alpha_1-\zeta) \nabla v_{n}\cdot \nabla\overline{p_{n}}dx
-\int_{0}^{\infty}\int_{\Omega}g(s)\nabla \eta_n\cdot \nabla \overline{p_n}dxds\to0.
\end{equation}

On the one hand, we claim that
\begin{equation}\label{435}
\int_{\Omega}\lambda^2_n\rho v_n \overline{p_n}dx\to0.
\end{equation}
Indeed, by (\ref{e1}) and (\ref{un}), we know that $\|\lambda_nv_n\|_{L^2(\Omega)}\to 0$, which together with H\"{o}lder inequality, (\ref{e3}) {and the boundedness of $\|q_n\|_{L^2(\Omega)}$}, yields (\ref{435}).

On the other hand,  using H\"{o}lder inequality again, along with \eqref{w0}, (\ref{vn}) and the boundedness of $\|\nabla p_n\|_{L^2(\Omega)}$, we obtain that the last two terms in (\ref{434}) satisfy  $\int_{\Omega}(\alpha_1-\zeta) \nabla v_{n}\cdot \nabla\overline{p_{n}}dx\to 0,$ and $\int_{0}^{\infty}\int_{\Omega}g(s)\nabla \eta_n\cdot \nabla \overline{p_n}dxds\to 0$.

Hence, by (\ref{434}), we get
\begin{equation}\nonumber
\int_{\Omega}\lambda^2_n\gamma \mu |p_n|^2dx\to 0,
\end{equation}
which along with (\ref{e3}) leads to
\begin{equation}\label{qn+}
\|q_n\|_{L^2(\Omega)}\to 0.
\end{equation}

\vspace{0.3cm}
\noindent{\it Step 5. $\|\gamma \nabla v_{n}-\nabla p_{n}\|_{L^2(\Omega)},\quad n\to \infty.$}
\vspace{0.2cm}

Taking the $L^2$-inner product  of $(\ref{ne2})$ and $(\ref{ne4})$ with $v_n$ and $p_n$, respectively, and using $(\ref{d-n})$, then integrating by parts yields that
\begin{align}
\label{ne2i} &-\int_{\Omega}\lambda^2_n\rho|v_n|^2 dx+\int_{\Omega}(\alpha-\zeta)|\nabla v_{n}|^2 dx-\int_{\Omega}\gamma\beta\nabla\overline{v_{n}}\cdot\nabla p_{n}dx\cr
&\qquad+\int_{0}^{\infty}\int_{\Omega}g(s)\nabla \eta_n\cdot \nabla \overline{v_n}dxds\rightarrow0,~\text{in}~L^2(\Omega),\\
\label{ne4i} &-\int_{\Omega} \lambda^2_n \mu|p_n|^2dx+\int_{\Omega}\beta|\nabla p_{n}|^2 dx-\int_{\Omega}\gamma\beta \nabla v_{n}\cdot\nabla\overline{p_{n}}dx\rightarrow0,~\text{in}~L^2(\Omega).
\end{align}

Adding $(\ref{ne2i})$ and $(\ref{ne4i})$, and  by virtue of $(\ref{e1})$, $(\ref{e3})$, \eqref{eta-v} and $\alpha_1=\alpha-\gamma^2\beta>0$, we obtain
{\begin{align}
\label{1} \int_{\Omega}(\alpha_1-\zeta)|\nabla v_{n}|^2 dx-\int_{\Omega}\rho|u_n|^2 dx+\int_{\Omega}\beta|\gamma\nabla v_{n}-\nabla p_{n}|^2dx-\int_{\Omega}\mu|q_n|^2 dx\rightarrow0.
\end{align}}

{Substituting $(\ref{un})$, $(\ref{vn})$ and \eqref{qn+} into $(\ref{1})$ gives that

\begin{align}\nonumber
\label{qn} & \int_{\Omega}\beta|\gamma \nabla v_{n}-\nabla p_{n}|^2dx\rightarrow0,
\end{align}
that is, $\|\gamma \nabla v_{n}-\nabla p_{n}\|_{L^2(\Omega)}\rightarrow 0$.

Summing up the above five steps, we have obtained  that $\|X_n\|_{\mathcal{H}}\rightarrow0$, which contradicts the fact $\|X_n\|_{\mathcal{H}}=1$. Thus, the condition (\ref{403}) in Lemma \ref{l-4-1} has been verified.

\vspace{0.5cm}
\noindent{\bf Part II.  We show that $i\mathbb{R}\subset\varrho(\mathcal{A})$.}

The proof by contradiction is still employed.
Suppose that $i\mathbb{R}\nsubseteq\varrho(\mathcal{A})$.
Thus, due to $0\in \varrho(\mathcal{A})$ and $\varrho(\mathcal{A})$ is an open set, we have
$$
0<\hat \lambda<\infty,
$$
in which
$$
\hat\lambda:=\sup\{R>0\; :\; [-Ri,Ri]\subset \varrho(\mathcal{A})\}.
$$
Thanks to Banach-Steinhaus theorem,  there exist  sequences $X_n=(v_n,u_n,p_n,q_n,\eta_n)\in \mathcal{D}(\mathcal{A})$ with $\|X_n\|_{\mathcal{H}}=1$ and $\lambda_n\rightarrow \hat{\lambda}$ such that
\begin{equation}\nonumber
(i\lambda_n-\mathcal{A})X_n\rightarrow 0~\text{in}~\mathcal{H},
\end{equation}
that is,  (\ref{e1})--(\ref{e5}) hold for $\lambda_n\rightarrow \hat{\lambda}$,
with the boundary conditions  $(\ref{d-n})$.

Note that {\it Step 1-5} in {\bf Part I} still holds for $\lambda_n\rightarrow \hat{\lambda}$.  Thus, following these five steps,  we can also achieve the contradiction   $\|X_n\|_{\mathcal{H}}=o(1)$, and hence $i\mathbb{R}\subset\varrho(\mathcal{A})$ holds.

Therefore, by {\bf Part I}, {\bf II} along with Lemma \ref{l-4-1}, the result in Theorem \ref{t-4-1} holds. The proof is completed.\hfill$\Box$

}

\subsection{Stability analysis of nonlinear system \eqref{sys2} or \eqref{eq2.8}}
\qquad
In this section, we  prove the exponential stability of the nonlinear system \eqref{sys2} with small initial data.
Inspired by \cite{freitas} and \cite{matf}, we assume that

{\color{blue}(H3)} There exists a {G\^{a}teaux differentiable functional $F:(H_{\Gamma_0}^1(\Omega))^2\rightarrow \mathbb{C}^2$ {satisfying $F(0,0)=0$} such that  $$\nabla F(v,p)=(f_1(v,p),f_2(v,p)),$$ where $f_1,f_2$ are the nonlinear source terms in system \eqref{sys1}. It  should be mentioned that $\nabla F$ denotes the unique vector representing the G\^{a}teaux  derivative $DF$ in the Riesz isomorphism.

Note that  Theorem \ref{t-4-1}  indicates that the semigroup $e^{\mathcal{A}t}$ associated with the linear part of system (\ref{sys2}) is exponentially stable on $\mathcal{H}$,  that is, there exist $\hat{M}>0$ and $\hat{\omega}>0$ such that
\begin{align}
\label{e-s}\|e^{\mathcal{A}t}\|_\mathcal{H}\leq \hat{M} e^{-\hat{\omega} t},~t\geq0.
\end{align}
Based on (\ref{e-s}), we have the following result on the stability of the nonlinear system (\ref{sys2}) with initial value satisfying $\|X_0\|_\mathcal{H}\leq \rho_0$, where $\rho_0>0$ is any given constant.
\begin{theorem}\label{th-stability}
Assume that {\color{blue}(H3)} holds, and $f_i,\; i=1,2$ satisfy {\color{blue}(H1) and (H2)} such that  for any $\|X\|_{\mathcal{H}},\|\widetilde{X}\|_{\mathcal{H}}\leq C_{\rho_0}$, there exists a constant  $L({C_{\rho_0}})<\frac{\hat{\omega}}{\hat{M}}$ satisfying
\begin{align}\label{loc-lip}
\|\mathcal{F}(X)-\mathcal{F}(\widetilde{X})\|_{\mathcal{H}}\leq L({C_{\rho_0}})\|X-\widetilde{X}\|_{\mathcal{H}}.
\end{align}
Then, the solution to system \eqref{sys2} decays to zero exponentially for every initial value $X_0\in \mathcal{H}$ satisfying $\|X_0\|_\mathcal{H}\leq \rho_0$.
\end{theorem}
}
\begin{proof}
We now define the total energy functional $E(t)$ of the system \eqref{sys2} by
\begin{align}\label{ene}
E(t)&=\frac{1}{2}\big[\int_{\Omega}(\alpha_1-\zeta)|\nabla v|^2dx+\int_{\Omega}\rho| u|^2dx+\int_{\Omega}\beta|\gamma \nabla v-\nabla p|^2dx+\int_{\Omega}\mu |q|^2dx\cr
&+\int_{0}^{\infty}\int_{\Omega}g(s)|\nabla\eta(x,s)|^2dxds\big]-\int_{\Omega}F(v,p)dx.
\end{align}

Note that, apart from the last term in (\ref{ene}), $E(t)$ is the natural energy for the linear part of \eqref{sys2}. {By taking the $L^2$-inner product of \eqref{sys2}$_1$ and \eqref{sys2}$_2$  with $v_t$ and $p_t$, respectively, then integrating by part and using $(\ref{d-n})$ yield that}
$$E'(t)=\int_{0}^{\infty}\int_{\Omega}g'(s)|\nabla \eta|^2dxds\leq 0,$$
which implies that $E(t)\leq E(0)$.

We mainly estimate the last term of $E(t)$. In view of the differential mean value inequality, {along with the assumption {\color{blue}(H3)}, we have
\begin{align}\label{d-m-v}
\big|\int_{\Omega}F(v,p)dx\big|&\leq \int_{\Omega}|\nabla F(\theta(v,p))||(v,p)|dx\cr
&\leq  c\int_{\Omega}(|f_1(\theta(v,p))|+|f_2(\theta(v,p))|)(|v|+|p|)dx,~\theta \in (0,1).
\end{align}

{Thanks to the assumption \eqref{li-gro} in {\color{blue}(H2)},  combining  \eqref{d-m-v}  with H\"{o}lder and  Poincar\'{e} inequality, we have
\begin{align}\nonumber
\big|\int_{\Omega}F(v,p)dx\big|&\leq 4c\int_{\Omega}(|v|^{r-1}+| p|^{r-1}+1)(|v|^2+| p|^2)dx\cr
&\leq 4cl_0\big(\int_{\Omega}|\nabla v|^2+|\nabla p|^2dx\big).
\end{align}
}} Then,
\begin{align}\label{d-m-v-fur2}
\big|\int_{\Omega}F(v,p)dx\big|&\leq 4cl_0\delta\int_{\Omega}\big[(\alpha_1-\zeta)|\nabla v|^2+\beta|\gamma\nabla v-\nabla p|^2\big]dx,
\end{align}
where $\delta=\max\{\frac{\gamma^2+1}{\alpha_1-\zeta},\frac{1}{\beta}\}$. Therefore, substituting \eqref{d-m-v-fur2} into \eqref{ene} shows that
\begin{align}\label{ene-2}
E(t)&\geq \frac{1}{2}\big[\int_{\Omega}(\alpha_1-\zeta)|\nabla v|^2dx+\int_{\Omega}\rho| u|^2dx+\int_{\Omega}\beta|\gamma \nabla v-\nabla p|^2dx+\int_{\Omega}\mu |q|^2dx\cr
&+\int_{0}^{\infty}\int_{\Omega}g(s)|\nabla\eta(x,s)|^2dxds\big]-4cl_0\delta\int_{\Omega}\big[(\alpha_1-\zeta)|\nabla v|^2+\beta|\gamma\nabla v-\nabla p|^2\big]dx.
\end{align}
Choosing $c<\frac{1}{8l_0\delta}$ and denoting $k_0=1-8cl_0\delta>0$, then it yields that
\begin{align}\label{ene-3}
E(t)&\geq \frac{k_0}{2}\|X\|_\mathcal{H}^2,~t\in \mathbb{R}_+.
\end{align}
Note that  by (\ref{d-m-v-fur2}) and (\ref{ene-2}), choosing $t=0$, we obtain that
\begin{equation}\label{451}
E(0)\leq\frac{\widetilde{k}_0}{2}\|X_0\|_\mathcal{H}^2, \quad \mbox{where}\quad \widetilde{k}_0=1+8cl_0\delta.
\end{equation} Thus, by (\ref{ene-3}) and  (\ref{451}), we have
{\begin{align}\nonumber
\|X\|_\mathcal{H}\leq \sqrt{\frac{\widetilde{k}_0}{k_0}}\|X_0\|_\mathcal{H}\leq\sqrt{\frac{\widetilde{k}_0}{k_0}}\rho_0=:C_{\rho_0},~t\in \mathbb{R}_+.
\end{align}
}
Let us recall \eqref{sol}, direct calculation yields that
\begin{eqnarray}\nonumber
\|X(t)\|_{\mathcal{H}}\leq \hat{M}e^{-\hat{\omega} t}\|X_0\|_{\mathcal{H}}+\hat{M}L(C_{\rho_0}) e^{-\hat{\omega} t}\int_{0}^{t}e^{\hat{\omega} s}\|X(s)\|_{\mathcal{H}}ds,
\end{eqnarray}
and hence,
\begin{eqnarray}\label{e-a+}
e^{\hat{\omega} t}\|X(t)\|_{\mathcal{H}}\leq \hat{M}\|X_0\|_{\mathcal{H}}+\hat{M}L(C_{\rho_0})\int_{0}^{t}e^{\hat{\omega} s}\|X(s)\|_{\mathcal{H}}ds.
\end{eqnarray}
{Applying} Gronwall's inequality to \eqref{e-a+}, we obtain
\begin{equation}\label{exp}
\|X(t)\|_{\mathcal{H}}\leq \hat{M}e^{-[\hat{\omega}-\hat{M}L(C_{\rho_0})]t}\|X_0\|_{\mathcal{H}},~t\in\mathbb{R}_+.
\end{equation}
\end{proof}
\begin{remark}
Based on the proof of Theorem \ref{th-stability}, the decay rate of the energy $E(t)$ (given as in (\ref{ene})) associated {\color{blue}with} the nonlinear system (\ref{sys2})  can also be estimated.
In fact, thanks to the estimation \eqref{d-m-v-fur2}, there exists a constant $\hat{C}>0$ such that
$
E(t)\leq \hat{C}\|X\|_{\mathcal{H}}^2,
$
which together with (\ref{exp}) implies the exponential decay of $E(t)$.
\end{remark}

\section{Conclusions}
\noindent

In this paper, we considered a kind of multi-dimensional nonlinear piezoelectric beam system subject to only one viscoelastic infinite memory in the elasticity equation. Under appropriate Hilbert settings and the suitable conditions on nonlinear source terms, the existence and uniqueness of global {solution} of the nonlinear system were shown by the semigroup theories and fixed point theorem. We further proved that its solution decays  exponentially  for small initial value based on frequency domain analysis and energy estimates. Especially, this decay rate is irrelevant to the relationships of the wave speeds $\frac{\rho}{\alpha}$ and $\frac{\mu}{\beta}$,  which is totally different from the well-known Timoshenko beam.

One promising research problem is to investigate the large time behaivor of the following abstract coupled system:
\begin{equation}\label{abstract}
\left\{
\begin{array}{ll}
 v_{tt}(t)+a_1A v(t)-\kappa A^\beta p(t)-\int_{0}^{\infty}g(s)A^\alpha v(t-s)ds=0,\\
p_{tt}(t)+a_2A p(t)-\kappa A^\beta v(t)=0,~\alpha,\beta\in[0,1),
\end{array}
\right.
\end{equation}
  where $A$ is a positive self-adjoint operator.  {When $\alpha=\beta=1$ and $A=-\Delta$, the system becomes a concrete one, that is, the linear part of system (\ref{sys1}).}  From this view, system (\ref{abstract}) is a more generalized case compared to system (\ref{sys1}). Note that the choices of $\alpha$ and $\beta$ are related to the memory properties and coupling properties, respectively. Based on the previous results on weakly coupled PDEs systems and single second order systems with memories, it is reasonable to predict that the coupling and memory properties can simultaneously affect the stability of the system (\ref{abstract}).  Thus, in a future work, we shall identify how the choices of  $\alpha$ and $\beta$ determine the decay rates of the solutions to system (\ref{abstract}).

{\subsection*{Acknowledgement}{\rm The authors would like to appreciate the AE and  referees  for their helpful and valuable comments and suggestions.}}
\subsection*{Statements and Declarations}{\rm The authors declare that they have no conflict of interest.}


\begin{thebibliography}{99}
	
	{\bibitem{Ammar} F. Ammar-Khodja, A. Benabdallah, J. E. Mu$\tilde{n}$oz Rivera and R. Racke, Energy decay for Timoshenko systems of memory type, \emph{J. Differ. Equations}, 194 (1) (2003) 82--115.
		
		\bibitem{alm} A. M. Al-Mahdi, M. M. Al-Gharabli, A. Guesmia and S. A. Messaoudi, New decay results for a viscoelastic-type Timoshenko system with infinite memory, \emph{Z. Angew. Math. Phys.}, 72 (2021) 22.}
	
	
	
	
	\bibitem{brez} H. Brezis, Functional Analysis, Sobolev Spaces and Partial Differential Equations, \emph{Springer}, New York, 2011.
	
	
	\bibitem{ch2} B. Chentouf and A. Guesmia, Well-posedness and stability results for the Korteweg-de Vries-Burgers and Kuramoto-Sivashinsky equations with infinite memory: A history approach, \emph{Nonlinear Anal. Real. World Appl.}, 65 (2022) 103508.
	
	
	\bibitem{cd} C. Dagdeviren, P. Joe, O. L. Tuzman, K. Park,
	K. J. Lee, Y. Shi, Y. G. Huang and J. A. Rogers, Recent progress in flexible and stretchable piezoelectric devices for mechanical energy harvesting, sensing and actuation, \emph{Extreme Mech. Lett.}, 9 (2016) 269--281.
	
	
	\bibitem{dafermos} C. Dafermos, Asymptotic stability in viscoelasticity, \emph{Arch. Ration. Mech. Anal.}, 37 (1970) 297--308.
	
	
	
	\bibitem{freitas} M. M. Freitas,  A. J. A. Ramos, A. \"{O}. \"{O}zer, and D. S. Almeida J\'{u}nior,  Long-time dynamics for a fractional piezoelectric system with magnetic effects and Fourier's law, \emph{J. Differ. Equations}, 280 (2021) 891--927.
	
	
	\bibitem{fre} M. M. Freitas, A. J. A. Ramos, M. J. D Santos and J. L. L Almeida, Dynamics of piezoelectric beams with magnetic effects and delay term, \emph{Evol. Equ. Control Theory}, 11 (2) (2022) 583--603.
	
	{\bibitem{Gear}
		L. Gearhart, Spectral theory for contraction semigroups on Hilbert space,
		{\it Trans. Amer. Math. Soc.}, 236 (1978) 385--394.}
	
	{\bibitem{huangfalun}
		F. L. Huang, Characteristic conditions for exponential stability of linear dynamical systems in Hilbert spaces, {\it Ann. Differential Equations}, 1 (1) (1985) 43--56.}
	
	
	\bibitem{Kh} H. E. Khochemane, A. Djebabla, S. Zitouni and L. Bouzettouta, Well-posedness and general decay of a nonlinear damping porous-elastic system with infinite memory, \emph{J. Math. Phys.}, 61 (2020) 021505.
	
	
	
	\bibitem{matf} T. F. Ma and R. N. Monteiro, Singular limit and long-time dynamics of Bresse systems, \emph{SIAM J. Math. Anal.}, 49 (4) (2017) 2468--2495.
	
	
	
	\bibitem{ma1} K. A. Morris and A. \"{O}. \"{O}zer, Strong stabilization of piezoelectric beams with magnetic effects, \emph{IEEE Conference on Decision $\&$ Control}, (2013) 3014--3019.
	
	\bibitem{ma2} K. A. Morris and A. \"{O}. \"{O}zer, Modeling and stabilizability of voltage-actuated piezoelectric beams with magnetic effects, \emph{ SIAM J. Control Optim.}, 52 (2014) 2371--2398.
	
	\bibitem{ma3} K. A. Morris and A. \"{O}. \"{O}zer, Comparison of stabilization of current-actuated and voltage-actuated piezoelectric beams,
	\emph{IEEE Conference on Decision $\&$ Control}, (2014) 571--576.
	
	\bibitem{mus} M. I. Mustafa and S. A. Messaoudi, General energy decay rates for a weakly damped Timoshenko system, \emph{J. Dyn. Control Syst.}, 16 (2) (2010) 211--226.
	
	\bibitem{sn1} S. Nakagiri, Optimal Control of Linear Retarded Systems in Banach Spaces, \emph{J. Math. Anal. Appl.}, 120 (1986) 169--210.
	
	\bibitem{pazy} A. Pazy, Semigroups of Linear Operators and Applications to Partial Differential Equations, \emph{Applied Mathematical Sciences}, Springer-Verlag, New York, 1983.
	
	
	\bibitem{pao} A. Paolucci and C. Pignotti, Exponential decay for semilinear wave equations with viscoelastic damping and delay feedback, \emph{Math. Control Signals Syst.}, 33 (2021) 617--636.
	
	{		\bibitem{pruss}
		J. Pr\"{u}ss, On the spectrum of $C_0$-semigroups, {\it Trans. Amer. Math. Soc.}, 284 (2) (1984) 847--857.}
	
	
	\bibitem{am4} A. J. A. Ramos, M. M. Freitas, D. S. Almeida J\'{u}nior, S. S. Jesus and T. R. S. Moura, Equivalence between exponential
	stabilization and boundary observability for piezoelectric beams with magnetic effect, \emph{Z. Angew. Math. Phys.}, 70 (2019) 60.
	
	
	
	\bibitem{anderson1} A. J. A. Ramos, Cledson S. L. Gon\c{c}alves
	and Silv\'{e}rio S. Correa N\^{e}to, Exponential stability and numerical treatment for
	piezoelectric beams with magnetic effect, \emph{ESAIM---Math. Model. Num.}, 52 (2018) 255--274.
	
	\bibitem{anderson2} A. J. A. Ramos, A. \"{O}. \"{O}zer, M. M. Freitas, D. S. Almeida J\'{u}nior and J. D. Martins, Exponential stabilization of fully dynamic and electrostatic piezoelectric beams with
	delayed distributed damping feedback, \emph{Z. Angew. Math. Phys.}, 72 (2021) 26.
	
	\bibitem{Rivera} J. E. Mu\~{n}oz Rivera and R. Racke, Mildly dissipative nonlinear Timoshenko systems---global existence and exponential
	stability, \emph{J. Math. Anal. Appl.}, 276 (2002) 248--278.
	
	
	\bibitem{smith} R. C. Smith, Smart material systems: Model Development, \emph{SIAM}, Philadelphia, 2005.
	
	
	
	\bibitem{gwess} M. Tucsnak and G. Weiss, Observation and control for operator semigroups, \emph{Birkh\"{a}user}, Basel-Boston-Berlin, 2009.
	
	
	
	{\bibitem{wangjunmin} J. M. Wang and B. Z. Guo, On dynamic behavior of a hyperbolic system derived
		from a thermoelastic equation with memory type,  \emph{J. Franklin Inst.}, 344 (2007) 75--96.}
	
	
	\bibitem{yang} J. Yang, A review of a few topics in piezoelectricity, \emph{Appl. Mech. Rev.}, 59 (2006) 335--345.
	
	
	
\end{thebibliography}
\end{document}